\documentclass[12pt]{article}
\usepackage{amssymb,latexsym,amsmath}
\usepackage{amsthm}
\usepackage{geometry}
\newtheorem{thm}{Theorem}[subsection]

\theoremstyle{plain}

\theoremstyle{definition}

\numberwithin {equation}{section}
\usepackage{longtable}

\begin{document}
\title
{New modification of Maheshwari method with optimal eighth order
of convergence for solving nonlinear equations}
\author{Somayeh Sharifi$^a$\thanks{s.sharifi@iauh.ac.ir} \and Massimiliano Ferrara$^b$\thanks{massimiliano.ferrara@unirc.it}
 \and Mehdi Salimi$^c$\thanks{Corresponding author: mehdi.salimi@tu-dresden.de}\and Stefan Seigmund$^c$\thanks{stefan.siegmund@tu-dresden.de}}
\date{}
\maketitle
\begin{center}
$^{a}$Young Researchers and Elite Club, Hamedan Branch, Islamic
Azad University, Hamedan, Iran\\
$^{b}$ Department of Law and Economics, University Mediterranea of Reggio Calabria, Italy\\
$^{c}$Center for Dynamics, Department of Mathematics, Technische
Universit{\"a}t Dresden, Germany
\end{center}
\date{}
\maketitle

\begin{abstract}
\noindent In this paper, we present a family of three-point with
eight-order convergence methods for finding the simple roots of
nonlinear equations by suitable approximations and weight function
based on Maheshwari method. Per iteration this method requires
three evaluations of the function and one evaluation of its first
derivative. This class of methods has the efficiency index equal
to $8^{\frac{1}{4}}\approx 1.682$. We describe the analysis of the
proposed methods along with numerical experiments including
comparison with existing methods.

\end{abstract}

\section{Introduction}
Finding roots of nonlinear functions $f(x)=0$ by using iterative
methods is a classical problem which has interesting applications
in different branches of science, in particular in physics and
engineering. Therefore, several numerical methods for
approximating simple roots of nonlinear equations have been
developed and analyzed by using various techniques based on
iterative methods in the recent years. The second order
Newton-Raphson method $x_{n+1} = x_n - \frac{f(x_n)}{f'(x_n)}$ is
one of the best-known iterative methods for finding approximate
roots and it requires two evaluations for each iteration step, one
evaluation of $f$ and one of $f'$ \cite{Ostrowski,Traub}.

Kung and Traub \cite{Kung} conjectured that no multi-point method
without memory with $n$ evaluations could have a convergence order
larger than $2^{n-1}$. A multi-point method with convergence order
$2^{n-1}$ is called optimal. The efficiency index provides a
measure of the balance between those quantities, according to the
formula $p^{1/n}$, where $p$ is the convergence order of the
method and $n$ is the number of function evaluations per
iteration.

Some well known two-point methods without memory are described by
e.g.\ Jarratt \cite{Jarratt}, King \cite{King}, Maheshwari
\cite{Mah} and Ostrowski \cite{Ostrowski}. Using inverse
interpolation, Kung and Traub \cite{Kung} constructed two general
optimal classes without memory. Since then, there have been many
attempts to construct optimal multi-point methods, utilizing e.g.\
weight functions \cite{Bi}, \cite{Chun1}, \cite{Cordero1},
\cite{Lotfi1}, \cite{Lotfi2}, \cite{Petkovic}, \cite{Salimi1},
\cite{Sharma1}, \cite{Thukral} and \cite{Wang}.

In this paper, we construct a new class of optimal eight order of
convergence based on Maheshwari method. This paper is organized as
follows: Section 2 is devoted to the construction and convergence
analysis of the new class. In Section 3, the new methods are
compared with a closest competitor in a series of numerical
examples, and Section 4 contains a short conclusion.

\section{Description of the method and convergence analysis}
\subsection{Three-point method of optimal order of convergence}
In this section we propose a new optimal three-point method based
on Maheshwari method \cite{Mah} for solving nonlinear equations.
The Maheshwari method is given by
\begin{equation}\label{a1}
\begin{cases}
y_{n}=x_{n}-\frac{f(x_{n})}{f'(x_{n})},\\
x_{n+1}=x_n+\frac{1}{f'(x_n)}\left(\frac{f^2(x_n)}{f(y_n)-f(x_n)}-\frac{f^2(y_n)}{f(x_n)}\right),\quad
(n=0, 1, \ldots),
\end{cases}
\end{equation}
where $x_0$ is an initial approximation of $x^{*}$. The
convergence order of (\ref{a1}) is four with three functional
evaluations per iteration such that this method is optimal. We
intend to increase the order of convergence of method (\ref{a1})
by an additional Newton step. So we obtain
\begin{equation}\label{a2}
\begin{cases}
y_{n}=x_{n}-\frac{f(x_{n})}{f'(x_{n})},\\
z_n=x_n+\frac{1}{f'(x_n)}\left(\frac{f^2(x_n)}{f(y_n)-f(x_n)}-\frac{f^2(y_n)}{f(x_n)}\right),\\
x_{n+1}=z_n-\frac{f(z_n)}{f'(z_n)}.
\end{cases}
\end{equation}
Method (\ref{a2}) uses five function evaluations with order of
convergence eight. Therefore, this method is not optimal. In order
to decrease the number of function evaluations, we approximate
$f'(z_n)$ by an expression based on $f(x_n)$, $f(y_n)$, $f(z_n)$
and $f'(x_n)$. Therefore

\begin{equation}
f'(z_n)\approx\frac{f'(x_n)}{F(x_n.y_n,z_n)H(s_n)},
\end{equation}
where\\
\begin{equation}\label{11}
F(x_n,y_n,z_n)=\left(\frac{f^3(y_n)(f(x_n)-10f(y_n))+4f^2(x_n)(f^2(y_n)+f(x_n)f(y_n))}{f(x_n)(2f(x_n)-f(y_n))^2(f(y_n)-f(z_n))}\right),
\end{equation}
\\
and $s_n=\frac{f(z_n)}{f(x_n)}$.\\
We have
\begin{equation}\label{c3}
\begin{cases}
y_{n}=x_{n}-\frac{f(x_{n})}{f'(x_{n})},\\
z_n=x_n+\frac{1}{f'(x_n)}\left(\frac{f^2(x_n)}{f(y_n)-f(x_n)}-\frac{f^2(y_n)}{f(x_n)}\right),\\
x_{n+1}=z_n-\frac{f(z_n)}{f'(x_n)}F(x_n,y_n,z_n)H(s_n),
\end{cases}
\end{equation}
where $F(x_n,y_n,z_n)$ and $s_n$ are defined as above.\\
\subsection{Convergence analysis}
In the following theorem we provide sufficient conditions on the
weight function $H(s_n)$ which imply that method (\ref{c3}) has
convergence order eight.

\begin{thm}\label{t1}
Assume that function $ f: D \rightarrow \mathbb{R}$ is eight times
continuously differentiable on an interval $D\subset\mathbb{R}$
and has a simple zero $x^{*}\in D$. Moreover, $H$ is one time
continuously differentiable. If the initial approximation $x_0$ is
sufficiently close to $x^{*}$ then the class defined by \eqref{c3}
converges to $x^{*}$ and the order of convergence is eight under
the conditions
\begin{equation*}
   \begin{split}
    H(0)&=1, \quad H^{'}(0)=2,
   \end{split}
\end{equation*}
with the error equation
\begin{equation*}
e_{n+1}=\left(\frac{1}{2}c_2^2(4c_2^2-c_3)(c_2^3-8c_2c_3+2c_4)\right)e_n^8+O(e_n^9).
\end{equation*}
\end{thm}
\begin{proof}
 Let $e_{n}:=x_{n}-x^{*}$, $e_{n,y}:=y_{n}-x^{*}$,
$e_{n,z}:=z_{n}-x^{*}$ and
$c_{n}:=\frac{f^{(n)}(x^{*})}{n!f^{'}(x^{*})}$ for $n\in \mathbb{N}$.
Using the fact that $f(x^{*})=0$, Taylor expansion of $f$ at
$x^{*}$ yields
\begin{equation}\label{a10}
f(x_{n}) = f^{'}(x^*)\left(e_{n} +
c_{2}e_{n}^{2}+c_{3}e_{n}^{3}+\cdots+c_{8}e_{n}^{8}\right)+O(e_n^{9}),
\end{equation}
and
\begin{equation}\label{a11}
f^{'}(x_{n}) = f^{'}(x^*)\left(1 +
2c_{2}e_{n}+3c_{3}e_{n}^{2}+4c_{4}e_{n}^{3}+\cdots+9c_9e_n^8\right)+O(e_n^{9}).
\end{equation}
Therefore
\begin{equation*}
\begin{split}
\frac{f(x_{n})}{f'(x_{n})}&=e_{n}-c_{2}e_{n}^{2}+\left(2c_{2}^{2}-2c_{3}\right)e_{n}^{3}+(-4c_2^3+7c_2c_3-3c_4)e_n^4\\
&+(8c_2^4-20c_2^2c_3+6c_3^2+10c_2c_4-4c_5)e_n^5\\
&+(-16c_2^5+52c_2^3c_3-28c_2^2c_4+17c_3c_4-c_2(33c_3^2-13c_5))e_n^6+O(e_n^{7}),
\end{split}
\end{equation*}
and hence
\begin{equation*}
\begin{split}
e_{n,y}&= y_n-x^*=c_{2}e_{n}^{2}+(-2c_2^2+2c_3)e_n^3+(4c_2^3-7c_2c_3+3c_4)e_n^4\\
&+(-8c_2^4+20c_2^2c_3-6c_3^2-10c_2c_4+4c_5)e_n^5\\
&+(16c_2^5-52c_2^3c_3+28c_2^2c_4-17c_3c_4+c_2(33c_3^2-13c_5))e_n^6+O(e_n^{7}).
\end{split}
\end{equation*}
For $f(y_n)$ we also have
\begin{equation}\label{a12}
f(y_{n}) = f^{'}(x^*)\left(e_{n,y} +
c_{2}e_{n,y}^{2}+c_{3}e_{n,y}^{3}+\cdots+c_{8}e_{n,y}^{8}\right)+O(e_{n,y}^{9}).
\end{equation}
Therefore  by substituting (\ref{a10}), (\ref{a11}) and
(\ref{a12}) into (\ref{a2}), we get
\begin{equation*}
\begin{split}
e_{n,z}&= z_n-x^*=(4c_2^3-c_2c_3)e_{n}^{4}+
(-27c_2^4+26c_2^2c_3-2c_3^2-2c_2c_4)e_n^5\\
&+\left(120c_2^5-196c_2^3c_3+39c_2^2c_4-7c_3c_4+c_2(54c_3^2-3c_5)\right)e_n^6+O(e_n^7).
\end{split}
\end{equation*}
For $f(z_n)$ we also get
\begin{equation}\label{a122}
f(z_{n}) = f^{'}(x^*)\left(e_{n,z} +
c_{2}e_{n,z}^{2}+c_{3}e_{n,z}^{3}+\cdots+c_{8}e_{n,z}^{8}\right)+O(e_{n,z}^{9}).
\end{equation}
From (\ref{a10}), (\ref{a12}) and (\ref{a122}) we obtain
\begin{equation}\label{a13}
\begin{split}
F(x_n,y_n,z_n)&=1+2c_2e_n+3c_3e_n^2+(-8c_2^3+2c_2c_3+4c_4)e_n^3+(\frac{83}{2}c_2^4-45c_2^2c_3+4c_3^2+3c_2c_4\\
&+5c_5)e_n^4+(\frac{7167}{8}c_2^6+1731c_2^4c_3-56c_3^3+429c_2^3c_4-245c_2c_3c_4+9c_4^2+c_2^2(815c_3^2\\
&-84c_5)+16c_3c_5)e_n^6+O(e_n^7).
\end{split}
\end{equation}
From (\ref{a10}) and (\ref{a122}) we have
\begin{equation}\label{a13}
\begin{split}
s_n&=(4c_2^3-c_2c_3)e_n^3+(-31c_2^4+27c_2^2c_3-2c_3^2-2c_2c_4)e_n^4+(151c_2^5\\
&-227c_2^3c_3+41c_2^2c_4-7c_3c_4+c_2(57c_3^2-3c_5))e_n^5+(592c_2^6+1266c_2^4c_3\\
&+38c_3^3-325c_2^3c_4+170c_2c_3c_4-6c_4^2-10c_3c_5+c_2^2(-608c_3^2+55c_5))e_n^6+O(e_n^7).
\end{split}
\end{equation}

Expanding $H$ at $0$ yields
\begin{equation}\label{a14}
H(s_n)=H(0)+H^{'}(0)s_n+O(s_n^{2}).
\end{equation}
Substituting (\ref{a10})-(\ref{a14}) into (\ref{c3}) we obtain
\begin{equation*}
e_{n+1}=x_{n+1}-x^*=R_4e_n^4+R_5e_n^5+R_6e_n^6+R_7e_n^7+R_8e_n^8+O(e_n^9),
\end{equation*}
where
\begin{equation*}
\begin{split}
R_4&=-c_2(4c_2^2-c_3)(-1+H(0)),\\
R_5&=0,\\
R_6&=0,\\
R_7&=-c_2^2(-4c_2^2+c_3)^2(-2+H'(0)).
\end{split}
\end{equation*}
By setting $R_4=R_7=0$ and $R_8\neq0$ the convergence order
becomes eight. Obviously
\begin{equation*}
\begin{split}
H(0)=1 \quad \Rightarrow \quad R_4&=0,\\
H'(0)=2 \quad \Rightarrow \quad R_7&=0,
\end{split}
\end{equation*}
consequently the error equation becomes
\begin{equation*}
e_{n+1}=\left(\frac{1}{2}c_2^2(4c_2^2-c_3)(c_2^3-8c_2c_3+2c_4)\right)e_n^8+O(e_n^9),
\end{equation*}
which finishes the proof of the theorem.
\end{proof}

In what follows we give some concrete explicit representations of
\eqref{c3} by choosing different weight function satisfying the
provided condition for the weight function $H(s_n)$ in Theorem
\ref{t1}.

\textbf{Method 1.} Choose the weight function $H$ as:
\begin{equation}\label{b11}
H(s)=1+2s.
\end{equation}
The function $H$ in (\ref{b11}) satisfies the assumptions of
Theorem \ref{t1} and we get
\begin{equation}\label{b1}
\begin{cases}
y_{n}=x_{n}-\frac{f(x_{n})}{f'(x_{n})},\\
z_n=x_n+\frac{1}{f'(x_n)}\left(\frac{f^2(x_n)}{f(y_n)-f(x_n)}-\frac{f^2(y_n)}{f(x_n)}\right),\\
x_{n+1}=z_n-\frac{f(z_n)}{f'(x_n)}(\frac{f(x_n)+2f(z_n)}{f(x_n)})F(x_n,y_n,z_n),
\end{cases}
\end{equation}
where $F(x_n,y_n,z_n)$ is evaluated by (\ref{11}).

\textbf{Method 2.} Choose the weight function $H$ as:
\begin{equation}\label{b22}
H(s)=\frac{1+4 s}{1+2s}.
\end{equation}
The function $H$  in (\ref{b22}) satisfies the assumptions of
Theorem \ref{t1} and we obtain
\begin{equation}\label{b2}
\begin{cases}
y_{n}=x_{n}-\frac{f(x_{n})}{f'(x_{n})},\\
z_n=x_n+\frac{1}{f'(x_n)}\left(\frac{f^2(x_n)}{f(y_n)-f(x_n)}-\frac{f^2(y_n)}{f(x_n)}\right),\\
x_{n+1}=z_n-\frac{f(z_n)}{f'(x_n)}(\frac{f(x_n)+4f(z_n)}{f(x_n)+2f(z_n)})F(x_n,y_n,z_n),
\end{cases}
\end{equation}
where $F(x_n,y_n,z_n)$ is evaluated by (\ref{11}).

\textbf{Method 3.} Choose the weight function $H$ as:
\begin{equation}\label{b33}
H(s)=\frac{1}{1-2 s}.
\end{equation}
The function $H$  in (\ref{b33}) satisfies the assumptions of
Theorem \ref{t1} and we get
\begin{equation}\label{b3}
\begin{cases}
y_{n}=x_{n}-\frac{f(x_{n})}{f'(x_{n})},\\
z_n=x_n+\frac{1}{f'(x_n)}\left(\frac{f^2(x_n)}{f(y_n)-f(x_n)}-\frac{f^2(y_n)}{f(x_n)}\right),\\
x_{n+1}=z_n-\frac{f(z_n)}{f'(x_n)}(\frac{f(x_n)}{f(x_n)-2f(z_n)})F(x_n,y_n,z_n),
\end{cases}
\end{equation}
where $F(x_n,y_n,z_n)$ is evaluated by (\ref{11}).

We apply the new methods (\ref{b1}), (\ref{b2}) and (\ref{b3}) to
several benchmark examples and compare them with existing
three-point methods which have the same order of convergence and
the same computational efficiency index equal to $\sqrt[\theta]{r}
= 1.682$ for the convergence order $r=8$ which is optimal for
$\theta = 4$ function evaluations per iteration \cite{Ostrowski,
Traub}.
\section{Numerical performance}
In this section we test and compare our proposed methods with some
existing methods. We compare methods (\ref{b1}), (\ref{b2}) and
(\ref{b3}) with the following related three-point methods.\\

\textbf{ Bi, Ren and Wu method.} The method by Bi et al.\
\cite{Bi} is given by
\begin{equation}\label{d1}
\begin{cases}
y_{n}=x_{n}-\frac{f(x_{n})}{f'(x_{n})},\\[0.7ex]
z_{n}=y_{n}-\frac{f(y_{n})}{f'(x_{n})}\cdot\frac{f(x_n)+\beta f(y_n)}{f(x_n)+(\beta-2)f(y_n)} ,\\[0.7ex]
x_{n+1}=z_{n}- \frac{f(z_n)}{f[z_n, y_n]+f[z_n, x_n,
x_n](z_n-y_n)}H(t),
\end{cases}
\end{equation}
where $\beta=-\frac{1}{2}$ and weight function
\begin{equation}\label{d2}
H(t)=\frac{1}{(1-\alpha t)^2}, \quad \alpha=1.
\end{equation}

\textbf{Wang and  Liu method.} The method by Wang and  Liu
\cite{Wang} is given by
\begin{equation}\label{d3}
\begin{cases}
y_n=x_n-\frac{f(x_n)}{f'(x_n)},\\[0.7ex]
z_n=x_n-\frac{f(x_n)}{f'(x_n)}~G(t),\\[0.7ex]
x_{n+1}=z_n-\frac{f(z_n)}{f'(x_n)}\left(H(t)+V(t)W(s)\right),
\end{cases}
\end{equation}
with weight functions
\begin{equation}\label{d4}
G(t)=\frac{1-t}{1-2t}, \quad H(t)=\frac{5-2t+t^2}{5-12t}, \quad
V(t)=1+4t, \quad W(s)=s.
\end{equation}

\textbf{Sharma and Sharma method.} The Sharma and  Sharma method
\cite{Sharma1} is given by
\begin{equation}\label{d5}
\begin{cases}
y_n=x_n-\frac{f(x_n)}{f'(x_n)},\\[0.7ex]
z_n=y_n-\frac{f(y_n)}{f'(x_n)}\cdot\frac{f(x_n)}{f(x_n)-2f(y_n)},\\[0.7ex]
x_{n+1}=z_n-\frac{f[x_n,y_n]f(z_n)}{f[x_n,z_n]f[y_n,z_n]}~W(t),
\end{cases}
\end{equation}
with weight function
\begin{equation}\label{d6}
W(t)=1+\frac{t}{1+\alpha t}, \quad \alpha=1.
\end{equation}

The three point method (\ref{c3}) is tested on a number of
nonlinear equations. To obtain a high accuracy and avoid the loss
of significant digits, we employed multi-precision arithmetic with
7000 significant decimal digits in the programming package of
Mathematica 8.\\
In order to test our proposed methods (\ref{b1}), (\ref{b2}) and
(\ref{b3}) and also compare them with the methods \eqref{d1},
\eqref{d3} and \eqref{d5} we choose the initial value $x_0$ using
the \texttt{Mathematica} command \texttt{FindRoot} \cite[pp.\
158--160]{Hazrat} and compute the error and the approximated
computational order of convergence, (ACOC) by the formula
\cite{acoc}
\[
\textup{ACOC}\approx\frac{\ln|(x_{n+1}-x_{n})/(x_{n}-x_{n-1})|}{\ln|(x_{n}-x_{n-1})/(x_{n-1}-x_{n-2})|}.
\]
\begin{table}[h!]
\begin{center}
\begin{tabular}{ c c c c c c c } \hline
M & W-F & $~~~~\vert x_{1} -x^{*} \vert~~~~$ & $~~~~\vert  x_{2} -
x^{*} \vert~~~~$ & $~~~~\vert
x_{3} - x^{*} \vert~~~~$ &$~~~~\vert x_{4} -x^{*} \vert~~~~$ & ACOC\\
\hline
$(\ref{b1})$ &$(\ref{b11})$ & $~~0.937e-8$ &$0.655e-63~~$ & $0.374e-504$ & $0.865e-1913$ & $8.0000$\\
$(\ref{b2})$ &$(\ref{b22})$ & $~~0.568e-4$ &$0.145e-30~~$ & $0.259e-243$ & $0.272e-1945$ & $8.0000$ \\
$(\ref{b3})$ &$(\ref{b33})$ & $~~0.755e-4$ &$0.141e-29~~$ & $0.206e-235$ & $0.423e-1882$ & $8.0000$ \\
$(\ref{d1})$ &$(\ref{d2})$ & $~~0.720e-4$ &$0.584e-30~~$ & $0.110e-238$ & $0.175e-1908$ & $8.0000$ \\
$(\ref{d3})$ &$(\ref{d4})$ & $~~0.278e-3$ &$0.779e-26~~$ & $0.296e-206$ & $0.128e-1649$ & $8.0000$\\
$(\ref{d5})$ &$(\ref{d6})$ & $~~0.753e-4$ &$0.619e-31~~$ & $0.128e-247$ & $0.453e-1981$ & $8.0000$\\
\hline
\end{tabular}\end{center}
\vspace*{-3ex} \caption{Comparison for
$f(x)=\ln(1+x^2)+e^{x^2-3x}\sin(x), x^{*}=0, x_0=0.35$, for
different methods (M) and weight functions (W-F).\label{table1}}
\end{table}
\hspace{1cm}
\begin{table}[h!]
\begin{center}
\begin{tabular}{ c c c c c c c } \hline
 M & W-F & $~~~~\vert x_{1} -x^{*} \vert~~~~$ & $~~~~\vert  x_{2} -
x^{*} \vert~~~~$ & $~~~~\vert
x_{3} - x^{*} \vert~~~~$ &$~~~~\vert x_{4} -x^{*} \vert~~~~$ & ACOC\\
\hline
$(\ref{b1})$ &$(\ref{b11})$ & $~~0.444e-11$ &$0.399e-94~~$ & $0.170e-758$ & $0.189e-6073$& $8.0000$\\
$(\ref{b2})$ &$(\ref{b22})$ & $~~0.445e-11$ &$0.404e-94~~$ & $0.187e-758$ & $0.394e-6073$& $8.0000$ \\
$(\ref{b3})$ &$(\ref{b33})$ & $~~0.443e-11$ &$0.395e-94~~$ & $0.155e-758$ & $0.902e-6077$& $8.0000$ \\
$(\ref{d1})$ &$(\ref{d2})$ & $~~0.423e-12$ &$0.134e-114~~$ & $0.445e-1037$ & $0.211e-9339$& $9.0000$ \\
$(\ref{d3})$ &$(\ref{d4})$ & $~~0.225e-11$ &$0.629e-96~~$ & $0.265e-773$ & $0.264e-6192$& $8.0000$\\
$(\ref{d5})$ &$(\ref{d6})$ & $~~0.172e-11$ &$0.581e-98~~$ & $0.984e-790$ & $0.663e-6324$& $8.0000$\\
\hline
\end{tabular}\end{center}
\vspace*{-3ex} \caption{Comparison for
$f(x)=\ln(1-x+x^2)+4\sin(1-x), x^{*}=1, x_0=1.1$, for different
methods (M) and weight functions (W-F).\label{table2}}
\end{table}
\hspace{1cm}
\begin{table}[h!]
\begin{center}
\begin{tabular}{ c c c c c c c } \hline
M & W-F & $~~~~\vert x_{1} -x^{*} \vert~~~~$ & $~~~~\vert  x_{2} -
x^{*} \vert~~~~$ & $~~~~\vert
x_{3} - x^{*} \vert~~~~$ &$~~~~\vert x_{4} -x^{*} \vert~~~~$ & ACOC\\
\hline
$(\ref{b1})$ &$(\ref{b11})$ & $~~0.783e-8$ &$0.648e-64~~$ & $0.142e-512$ & $0.765e-4102$& $8.0000$\\
$(\ref{b2})$ &$(\ref{b22})$ & $~~0.749e-8$ &$0.656e-64~~$ & $0.855e-514$& $0.132e-4111$& $8.0000$ \\
$(\ref{b3})$ &$(\ref{b33})$ & $~~0.816e-8$ &$0.908e-64~~$ & $0.212e-511$& $0.187e-4092$& $8.0000$ \\
$(\ref{d1})$ &$(\ref{d2})$ & $~~0.673e-8$ &$0.113e-64~~$ & $0.726e-519$& $0.208e-4152$& $8.0000$ \\
$(\ref{d3})$ &$(\ref{d4})$& $~~0.997e-10$ &$0.751e-80~~$ & $0.782e-641$& $0.107e-5128$& $8.0000$\\
$(\ref{d5})$ &$(\ref{d6})$& $~~0.642e-10$ &$0.101e-81~~$ & $0.389e-656$& $0.184e-5251$& $8.0000$\\
\hline
\end{tabular}\end{center}
\vspace*{-3ex} \caption{Comparison for
$f(x)=x^4+\sin(\frac{\pi}{x^2})-5, x^{*}=\sqrt{2}, x_0=1.5$, for
different methods (M) and weight functions (W-F).\label{table3}}
\end{table}
\hspace{1cm}
\begin{table}[h!]
\begin{center}
\begin{tabular}{ c c c c c c c} \hline
M & W-F & $~~~~\vert x_{1} -x^{*} \vert~~~~$ & $~~~~\vert  x_{2} -
x^{*} \vert~~~~$ & $~~~~\vert
x_{3} - x^{*} \vert~~~~$ &$~~~~\vert x_{4} -x^{*} \vert~~~~$ & ACOC\\
\hline
$(\ref{b1})$ &$(\ref{b11})$ & $~~0.119e-3$ &$0.253e-26~~$ & $0.106e-207$ & $0.992e-1659$& $8.0000$\\
$(\ref{b2})$ &$(\ref{b22})$ & $~~0.143e-3$ &$0.109e-25~~$ & $0.124e-202$& $0.353e-1618$& $8.0000$ \\
$(\ref{b3})$ &$(\ref{b33})$ & $~~0.916e-4$ &$0.307e-27~~$ & $0.493e-215$& $0.221e-1717$& $8.0000$ \\
$(\ref{d1})$ &$(\ref{d2})$ & $~~0.183e-4$ &$0.319e-33~~$ & $0.278e-263$& $0.920e-2104$& $8.0000$ \\
$(\ref{d3})$ &$(\ref{d4})$& $~~0.612e-4$ &$0.111e-28~~$ & $0.134e-220$& $0.588e-1810$& $7.9999$\\
$(\ref{d5})$ &$(\ref{d6})$& $~~0.239e-4$ &$0.138e-32~~$ & $0.170e-258$& $0.938e-2066$& $8.0000$\\
\hline
\end{tabular}\end{center}
\vspace*{-3ex} \caption{Comparison for
$f(x)=(x-2)(x^{10}+x+1)e^{-x-1}, x^{*}=2, x_0=2.1$, for different
methods (M) and weight functions (W-F).\label{table4}}
\end{table}
\\
\\
\\
\\
\\
\\
\\
\\
\\
\\
\\
\\
\\
\\
In Tables \ref{table1}, \ref{table2}, \ref{table3} and
\ref{table4}, the proposed methods (\ref{b1}), (\ref{b2}) and
(\ref{b3}) with the methods (\ref{d1}), (\ref{d3}) and (\ref{d5})
have been tested on
 different nonlinear equations. It is clear that these
methods are in accordance with the developed theory.

\section{Conclusion}
We presented a new optimal class of three-point methods without
memory for approximating a simple root of a given nonlinear
equation. Our proposed methods use five function evaluations for
each iteration. Therefore they support Kung and Traub's
conjecture. Numerical examples show that our methods work and can
compete with other methods in the same class.


\begin{thebibliography}{99}

\bibitem{Bi} Bi, W., Ren, H., Wu, Q., Three-step iterative methods with
eighth-order convergence for solving nonlinear equations, J.
Comput. Appl. Math., 225, (2009), 105-112.

\bibitem{Chun1} Chun, C., Lee, M.Y., A new optimal eighth-order
family of iterative methods for the solution of nonlinear
equations, Appl. Math. Comput., 223, (2013), 506-519.

\bibitem{acoc}
Cordero, A., Torregrosa, J.R., Variants of Newton's method using
fifth-order quadrature formulas, Appl. Math. Comput., 190, (1),
686-698, (2007).

\bibitem{Cordero1} Cordero, A., Fardi, M., Ghasemi, M., Torregrosa, J.R., Accelerated
iterative methods for finding solutions of nonlinear equations and
their dynamical behavior, Calcolo, (2012), 1-14.

\bibitem{Hazrat}
 Hazrat, R., Mathematica�: A Problem-Centered Approach, Springer-Verlag, 2010.

\bibitem{Jarratt}
Jarratt, P., Some fourth order multipoint iterative methods for
solving equations, Math. Comput., 20, (1966), 434-437.

\bibitem{King}
King, R.F., Family of four order methods for nonlinear equations,
SIAM, Numer. Anal., 10, (1973), 876-879.

\bibitem{Kung}
 Kung, H.T., Traub, J.F., Optimal order of one-point and multipoint
  iteration, Assoc. Comput. Math., 21, (1974), 634-651.

\bibitem{Lotfi1}
Lotfi, T., Sharifi, S., Salimi, M., Siegmund, S., A new class of
three-point methods with optimal convergence order eight and its
dynamics, Numer. Algor., Published Online, (2014).


\bibitem{Lotfi2} Lotfi, T., Salimi, M., A note on the paper "A family of optimal iterative methods with fifth and tenth order convergence
for solving nonlinear equations", Communications in Numerical
Analysis 2013, 1-3, doi: 10.5899/2013/cna-00185, (2013).

\bibitem{Mah} Maheshwari, A.K., A fourth-order iterative method for solving
nonlinear equations, Appl. Math. Comput., 211, (2009),
383-391.


\bibitem{Ostrowski}
Ostrowski, A.M., Solution of Equations and Systems of Equations,
Academic Pres, New York, 1966.

\bibitem{Petkovic} Petkovic, M. S., Neta, B., Petkovic, L. D.,
Dzunic, J., Multipoint Methods for Solving Nonlinear Equations,
Elsevier, Waltham, MA, 2013.


\bibitem{Salimi1} Salimi, M., Lotfi, T., Sharifi, S., Siegmund, S., Optimal Newton-Secant like methods without memory for solving nonlinear
equations, Numer. Algor., Under Revision, (2014).


\bibitem{Sharma1}
Sharma, J.R., Sharma, R., A new family of modified Ostrowski's
methods with accelerated eighth order convergence, Numer. Algor.,
54, (2010), 445-458.

\bibitem{Traub}
Traub, J.F., Iterative Methods for the Solution of Equations,
Prentice Hall, New York, 1964.

\bibitem{Thukral}Thukral, R., Petkovic, M.S., A family of
three-point methods of optimal order for solving nonlinear
equations, J. Comput. Appl. Math., 233, (2010), 2278-2284.

\bibitem{Wang}
Wang, X., Liu, L., New eighth-order iterative methods for solving
nonlinear equations, J. Comput. Appl. Math., 234, (2010),
1611-1620.

\end{thebibliography}
\end{document}